\documentclass[12pt,reqno]{amsart}
\usepackage{amsmath}
\usepackage{amssymb}
\usepackage{amstext}
\usepackage{mathrsfs}
\usepackage{a4wide}
\usepackage{graphicx}
\usepackage{bbm}
\usepackage{verbatim}
\usepackage{bm}
\usepackage[unicode=true,
bookmarks=false,
breaklinks=false,pdfborder={0 0 1},colorlinks=true,linkcolor=blue,anchorcolor=blue,citecolor=blue]{hyperref}

\allowdisplaybreaks \numberwithin{equation}{section}
\usepackage{color}
\usepackage{cases}

\numberwithin{equation}{section}

\newtheorem{theorem}{Theorem}[section]

\newtheorem*{thA}{Theorem A}

\newtheorem{lemma}[theorem]{Lemma}

\theoremstyle{definition}
\newtheorem{remark}[theorem]{Remark}

\theoremstyle{remark}

\newtheorem{example}[theorem]{Example}

\def\d{\mathrm{d}}

\newcommand{\mv}{\mathbf{v}}
\newcommand{\R}{\mathbb{R}}

\begin{document}
\title[Free boundary problems]{Free boundary problems for the two-dimensional Euler equations in exterior domains}

\author{Daomin Cao, Boquan Fan,  Weicheng Zhan}

\address{Institute of Applied Mathematics, Chinese Academy of Sciences, Beijing 100190, and University of Chinese Academy of Sciences, Beijing 100049,  P.R. China}
\email{dmcao@amt.ac.cn}
\address{Institute of Applied Mathematics, Chinese Academy of Sciences, Beijing 100190, and University of Chinese Academy of Sciences, Beijing 100049,  P.R. China}
\email{fanboquan22@mails.ucas.ac.cn}
\address{School of Mathematical Sciences, Xiamen University, Xiamen, Fujian, 361005, P.R. China}
\email{zhanweicheng@amss.ac.cn}

	\begin{abstract}
In this paper we present some classification results for the steady Euler equations in two-dimensional exterior domains with free boundaries. We prove that, in an exterior domain, if a steady Euler flow devoid of interior stagnation points adheres to slip boundary conditions and maintains a constant norm on the boundary, along with certain additional conditions at infinity, then the domain is the complement of a disk, and the flow is circular, namely the streamlines are concentric circles. Additionally, we establish that in the entire plane, if all the stagnation points of a steady Euler flow coincidentally form a disk, then, under certain additional reasonable conditions near the stagnation points and at infinity, the flow must be circular. The proof is based on a refinement of the method of moving planes.
	\end{abstract}
	
	\maketitle{\small{\bf Keywords:} The Euler equations, Free boundary problems, Symmetry results, Exterior domains. \\

	
	\section{Introduction and Main results}

Let $\Omega\subset \R^2$ be a simply connected bounded domain of class $C^2$ and $n$ be the outward unit normal on $\partial \Omega$. Let us consider steady incompressible flows solving the Euler equations
\begin{equation}\label{1-1}
  \left\{
\begin{aligned}
  \mathbf{v}\cdot \nabla \mathbf{v} +\nabla P & =0&\text{in}\ \, \R^2\backslash \overline{\Omega},\\
  \nabla\cdot\mathbf{v} & =0\,\ \, \ \ \ \ \ \ \  \ \, &\text{in}\ \, \R^2\backslash \overline{\Omega},\\
   \mathbf{v}\cdot n &  =0\,\ \, \ \ \ \ \ \ \  \ \,&\text{on}\ \, \partial \Omega,\ \ \,\\
\end{aligned}
\right.
\end{equation}
where $\mathbf{v}=(v_1, v_2)$ is the velocity field and $P$ is the scalar pressure. Throughout this paper, the solutions are always understood in the classical sense, that is, both $\mathbf{v}$ and $P$ are (at least) of class $C^1$ and satisfy \eqref{1-1} pointwise.

We are concerned in this paper with rigidity properties of a steady Euler flow in two-dimensional exterior domains whose boundary is free but on which the flow is assumed to satisfy an additional condition. More precisely, we are interested in the solutions ($\mv, P$) of the Euler equations \eqref{1-1} that additionally $|\mv|$ is a constant on $\partial \Omega$. In general, this additional condition will make system \eqref{1-1} overdetermined. The investigation of similar Serrin-type free boundary problems in smooth simply or doubly connected bounded domains was recently considered by Hamel and Nadirashvili \cite{HN}. Under certain additional reasonable conditions, they showed that the domains must be disks or annuli, and the flow must be circular, namely the streamlines are concentric circles; see Theorems 1.10 and 1.13 in \cite{HN}. Further rigidity results on bounded domains can be found in \cite{Ruiz, Wang}.

Note that the problems addressed in the present paper are set in unbounded domains. It seems that rigidity results for unbounded domains are much more difficult to obtain; see Section 4 in \cite{Pie} for some discussions on this aspect. The aim of this paper is to establish some classification results for the steady Euler equations in two-dimensional exterior domains with free boundaries. To state our results we need some notation. Let $B_R(x)$ be the open disk with center $x$ and radius $R$, and for simplicity, we denote $B_R=B_R(0)$. Let us denote
\begin{equation*}
  \mathbf{e}_r(x)=\frac{x}{|x|}\ \ \ \text{and}\ \ \  \mathbf{e}_\theta(x)=\left(-\frac{x_2}{|x|}, \frac{x_1}{|x|}\right)
\end{equation*}
for $x=(x_1, x_2)\in \R^2\backslash\{0\}$.

\smallskip
Our first main result is the following theorem.

\begin{theorem}\label{th3}
  Let $\mathbf{v}$ be a $C^2({\R^2\backslash\Omega})$ flow solving \eqref{1-1}. Assume also that:
  \begin{itemize}
    \item [(1)]$|\mathbf{v}|>0$ in $\R^2\backslash \overline{\Omega}$ and $|\mathbf{v}|$ is a nonzero constant on $\partial \Omega$.
    \item [(2)]The following far-field conditions hold
    \begin{equation}\label{1-3}
      \liminf_{|x|\to +\infty}|\mathbf{v}(x)|>0\ \ \ \text{and}\ \ \ \mathbf{v}(x)\cdot \mathbf{e}_r(x)=o\left(\frac{1}{|x|} \right)\ \text{as}\ |x|\to +\infty.
    \end{equation}
  \end{itemize}
Then $ \Omega=B_R$ for some $R>0$. Furthermore, $\mathbf{v}$ is a circular flow, that is, there is a $C^2(R,+\infty)$ function $V: (R, +\infty)\to \R$ with constant strict sign such that
  \begin{equation*}
    \mathbf{v}(x)=V(|x|)\mathbf{e}_\theta(x)
  \end{equation*}
  for all $|x|>R$.
\end{theorem}

Theorem \ref{th3} provides a classification result for the steady Euler equations in two-dimensional exterior domains with free boundaries. Under the hypothesis of Theorem \ref{th3}, it follows that the only admissible solution is characterized by a circular flow. Below is an example of such flow configurations.
\begin{example}\label{ex1}
   The smooth flow given by
  \begin{equation*}
    \mv(x)=|x|\mathbf{e}_\theta(x),\ \ \ P(x)=|x|^2/2
  \end{equation*}
clearly solves \eqref{1-1} with $\Omega=B_1$ and satisfies $|\mathbf{v}|=1$ on $\partial \Omega$.
\end{example}

Given that the problem under consideration is set in an unbounded domain, suitable far-field conditions are typically necessary, although the condition \eqref{1-3} presented here might not be optimal. The identical far-field condition is also featured in the recent study by Hamel and Nadirashvili; see Theorem 1.3 in \cite{HN} (see also Theorem A below).

We also note that the existence of nontrivial exterior domains where the Serrin-type free boundary problem admits a nontrivial solution was established by Ros-Ruiz-Sicbaldi; see Theorem 1.1 in \cite{Ros}. However, the counterexample constructed in \cite{Ros} does not satisfy the conditions of Theorem \ref{th3}. Indeed, it is imperative to note that the stream function of the examined flow must be unbounded. For more details, refer to Section \ref{s2} below.

Notice that in Theorem \ref{th3}, we impose the condition that $|\mv|$ remains a nonzero constant on the boundary. An interesting question to ask is whether the conclusion of Theorem \ref{th3} still holds true when all boundary points of the domain are stagnation points, specifically, when $|\mv|=0$ on $\partial \Omega$. Our second main result gives a positive answer to this question, with some additional assumptions.

\begin{theorem}\label{th2}
  Let $\mathbf{v}$ be a $C^2({\R^2\backslash\Omega})$ flow solving \eqref{1-1}. Assume also that:
  \begin{itemize}
    \item [(1)]$|\mathbf{v}|>0$ in $\R^2\backslash \overline{\Omega}$ and $|\mv|=0$ on $\partial \Omega$.
       \item [(2)]The following far-field conditions hold
    \begin{equation*}
      \liminf_{|x|\to +\infty}|\mathbf{v}(x)|>0\ \ \ \text{and}\ \ \ \mathbf{v}(x)\cdot \mathbf{e}_r(x)=o\left(\frac{1}{|x|} \right)\ \text{as}\ |x|\to +\infty.
    \end{equation*}
    \item [(3)]The vorticity $\omega:=\partial_1v_2-\partial_2 v_1$ does not vanish on $\partial \Omega$.
  \end{itemize}
 Then, the conclusion of Theorem \ref{th3} remains valid.
\end{theorem}

In comparison to Theorem \ref{th3}, our additional assumption entails the non-vanishing vorticity of the flow on $\partial\Omega$. We would like to mention that this assumption is primarily introduced to mitigate the degeneration near stagnation points of the flow. Here the stagnation points of a flow $\mathbf{v}$ are the points $x$ for which $|\mv(x)|=0$. The condition wherein all boundary points act as stagnation points effectively imposes no-slip boundary conditions on steady Euler flows, suggesting that fluid particles adhere to the boundary $\partial \Omega$ from a physical perspective. This degenerate situation poses several challenges for mathematical analysis.

Let $\omega:=\partial_1v_2-\partial_2 v_1$ represent the vorticity of the flow. It directly ensues from \eqref{1-1} that
\begin{equation*}
  \nabla \omega\cdot \mv=0.
\end{equation*}
It follows that the vorticity remains constant along the streamlines, thereby inferring that $\omega$ attains constancy on $\partial \Omega$, since $\partial \Omega$ serves as a streamline. So our assumption regarding the vorticity is equivalent to saying that $\omega$ is a nonzero constant on $\partial \Omega$.  Below is an example of such flow configurations.
\begin{example}\label{ex2}
  The smooth flow given by
  \begin{equation*}
    \mv(x)=2(|x|-1)\mathbf{e}_\theta(x),\ \ \ P(x)=4\ln|x|-4|x|
  \end{equation*}
clearly solves \eqref{1-1} with $\Omega=B_1$ and satisfies $|\mathbf{v}|=0$ on $\partial \Omega$. The vorticity $\omega$ is given by
\begin{equation*}
  \omega=2+{2(|x|-1)}/{|x|},
\end{equation*}
which evaluates to $2$ on $\partial B_1$.
\end{example}

 Recently, Ruiz \cite{Ruiz} established some rigidity results for nonzero compactly supported steady solutions of the Euler equations in the whole plane. Specifically, he demonstrated that if the non-stagnant points of the flow coincidentally form either a annular-shaped domain or a punctured simply connected domain, then the flow exhibits circular characteristics. More robust results are also established in the work of Wang and Zhan \cite{Wang}.

Let us turn now to consider steady solutions of the Euler equations in the whole plane:
\begin{equation}\label{1-5}
  \left\{
\begin{aligned}
  \mathbf{v}\cdot \nabla \mathbf{v} +\nabla P & =0&\text{in}\ \, \R^2,\\
  \nabla\cdot\mathbf{v} & =0\,\ \, \ \ \ \ \ \ \  \ \, &\text{in}\ \, \R^2.\\
\end{aligned}
\right.
\end{equation}
 Let
\begin{equation*}
  \mathcal{S}=\left\{x\in \R^2: |\mathbf{v}(x)|=0 \right\}
\end{equation*}
be the set of stagnation points of the flow $\mv$ in $\R^2$. An intriguing question arises as to whether $\mathbf{v}$ constitutes a circular flow in the event that the set $\mathcal{S}$ happens to coincide with a disk. We could not find in the literature any rigidity results concerning this matter. The third objective of this paper is to address this question. We will give an affirmative answer to this question under certain supplementary conditions.

Let us denote
\begin{equation*}
  \mv(x)=v^r(x)\mathbf{e}_r(x)+v^\theta(x) \mathbf{e}_\theta(x).
\end{equation*}
Here, $v^r(x)$ and $v^\theta(x)$ are referred to as the radial velocity and angular velocity of $\mv(x)$, respectively. It is noteworthy that the far-field condition \eqref{1-3} can be expressed equivalently in the following form
\begin{equation*}
   \liminf_{|x|\to +\infty}v^\theta(x)>0\ \ \ \text{and}\ \ \ v^r(x)=o\left(\frac{1}{|x|} \right)\ \text{as}\ |x|\to +\infty
\end{equation*}
 after potentially changing $\mv$ into $-\mv$. Our last result reads as follows.

  \begin{theorem}\label{th1}
 Let $\mathbf{v}$ be a $C^2(\R^2)$ flow solving \eqref{1-5}. Assume also that:
  \begin{itemize}
    \item [(1)]$\mathcal{S}=B_R$ for some $R>0$.
  \item [(2)]The following far-field conditions hold
    \begin{equation*}
          \liminf_{|x|\to +\infty}v^\theta(x)>0\ \ \ \text{and}\ \ \ v^r(x)=o\left(\frac{1}{|x|} \right)\ \text{as}\ |x|\to +\infty.
    \end{equation*}
       \item [(3)]There exists a $\delta>0$ such that $v^\theta(x)>0$ for all $R<|x|<R+\delta$, and $v^r(x)=o\left(v^\theta(x)\right)$ as $|x|\to R^+$.
  \end{itemize}
 Then $\mathbf{v}$ is a circular flow, that is, there is a $C^2(R,+\infty)$ positive function $V: (R, +\infty)\to \R$ such that
  \begin{equation*}
    \mathbf{v}(x)=V(|x|)\mathbf{e}_\theta(x)
  \end{equation*}
  for all $|x|>R$.
\end{theorem}

Theorem \ref{th1} is an example of a Liouville theorem for steady solutions of the Euler equations in the whole plane. It shows that any smooth steady solutions that satisfy the conditions in Theorem \ref{th1} must be circular flows, isolating such configurations from non-circular steady states.

Below is an example of such flow configurations.
\begin{example}\label{ex1}
  One can readily verify that the flow given by
  \begin{equation*}
    \mv(x)=4(|x|-1)^3_+\mathbf{e}_\theta(x),\ \ \ P(x)=-\frac{|\mathbf{v}|^2}{2}-F\left((|x|-1)_+^4 \right)
  \end{equation*}
 solves \eqref{1-5}, where
 \begin{equation*}
  t_+:=\max\{t, 0\}\ \ \ \text{and}\ \ \ F(t):=-4\int_{0}^{t} \left(3\sqrt{s}+\frac{\sqrt[4]{s^3}}{1+\sqrt[4]{s}}  \right)\d s.
 \end{equation*}
 Clearly, $|\mathbf{v}|=0$ in $B_1(0)$, and the conditions (2) and (3) in Theorem \ref{th1} are also fulfilled.
\end{example}

Several further comments are in order. First of all, it is worth pointing out condition (1) in Theorem \ref{th1} implies that the vorticity of the flow automatically vanishes on the boundary of $\mathcal{S}$. Hence, condition (3) in Theorem \ref{th2} becomes untenable in this scenario, rendering Theorem \ref{th2} inapplicable. We point out that condition (3) in Theorem \ref{th1} is primarily introduced to alleviate the severe degeneration near stagnation points of the flow. It serves to partly ensure the circularity of the flow near the periphery of $\mathcal{S}$, behaving similarly to that at infinity, with counterclockwise motion. This effectively guarantees that the flow $\mv$ does not exhibit excessive disorderliness near the periphery of $\mathcal{S}$. It is not clear to us whether such a restriction is only technical or deep-seated into the problem.


On the other hand, in \cite{HN}, Hamel and Nadirashvili established the following theorem (stated equivalently below).
\begin{thA}[\cite{HN}, Theorem 1.3]\label{thA}
Let $\mv$ be a $C^2(\R^2\backslash B_R)$ flow solving \eqref{1-1} with $\Omega=B_R$ for some $R>0$.
 Assume also that:
   \begin{itemize}
    \item [(i)] $\{x\in {\R^2\backslash B_R}: |\mathbf{v}(x)|=0\}\subsetneq \partial B_R$.
    \item [(ii)]The following far-field conditions hold
    \begin{equation*}
          \liminf_{|x|\to +\infty}v^\theta(x)>0\ \ \ \text{and}\ \ \ v^r(x)=o\left(\frac{1}{|x|} \right)\ \text{as}\ |x|\to +\infty.
    \end{equation*}
  \end{itemize}
Then $\mathbf{v}$ is a circular flow, that is, there is a $C^2[R,+\infty)$ positive function $V: [R, +\infty)\to \R$ such that
  \begin{equation*}
    \mathbf{v}(x)=V(|x|)\mathbf{e}_\theta(x)
  \end{equation*}
  for all $|x|\ge R$.
\end{thA}
We would like to point out that the non-stagnation condition (i) in Theorem A appears to be essential in the argument presented in \cite{HN}. The effect of this condition is twofold. Firstly, it ensures the absence of interior stagnation points, which is crucial for proving that the stream function of the flow satisfies a semi-linear elliptic equation. Secondly, it prevents the flow from stagnating along the entire boundary of the domain, thereby facilitating the implementation of the moving plane scheme. The case where the boundary points are all stagnant points can be regarded as a degenerate case. This degenerate case is out of the scope of Theorem A. Our Theorems \ref{th2} and \ref{th1} deal mainly with such a degenerate case and can therefore be taken to be a complement of Theorem A.

Lastly, we mention that it would be interesting to construct a non-circular flow in the whole plane such that its stagnation points form precisely a disk. It is also intriguing to construct counterexamples to the rigidity results established in \cite{HN3, HN1, HN, Ruiz, Wang} in scenarios where specific conditions are not met. Indeed, several counterexamples have recently been rigorously constructed; see, e.g., \cite{Ago, Kam, Pie, Ros, Ruiz2}.

\subsection*{Comments on related works}
In recent years, there has been tremendous interest in investigating the rigidity in steady motion of an ideal fluid. Kalish \cite{Kal} first showed that steady flows $\mv$ in the two-dimensional strip $\R\times (0, 1)$ satisfying $v_1\not=0$ are necessarily parallel shear flows. In \cite{HN3}, Hamel and Nadirashvili further proved that the same conclusion also holds under the non-stagnation condition $|\mv|\not=0$. Moreover, rigidity results for parallel shear flows under the non-stagnation condition $|\mv|\not=0$ also hold in two-dimensional half-planes \cite{HN3, HN2} and in the whole plane \cite{HN1}. We refer the interested readers to \cite{Bed, Cons, Coti, HN0, Ion, Li, Zill} and the relevant references therein for more rigidity/non-rigidity results and additional stability analyses pertaining to parallel shear flows.

In \cite{HN}, Hamel and Nadirashvili considered steady Euler flows in bounded annuli, as well as in complements of disks, in punctured disks and in the punctured plane. They showed that if the flow does not have any stagnation point and satisfies the tangential boundary conditions together with further conditions at infinity in the case of unbounded domains and at the center in the case of punctured domains, then the flow is circular. Recently, Wang and Zhan \cite{Wang} improved upon certain rigidity results discussed in \cite{HN} through the utilization of a rearrangement technique known as continuous Steiner symmetrization, developed by Brock in \cite{Bro0, Bro1}. They showed that in bounded annuli, steady Euler flows must be circular flows provided they have no interior stagnation point. Moreover, in disks, steady Euler flows with only one interior stagnation point and tangential boundary conditions must also be circular flows.

As mentioned earlier, Ruiz \cite{Ruiz} established some symmetry results for steady solutions of the two-dimensional Euler equations with compact support. Specifically, he showed that if the non-stagnant points coincidentally form either a annular-shaped domain or a punctured simply connected domain, then the flow must be circular flows. More robust results are also established in the work of Wang and Zhan \cite{Wang}.

Serrin-type free boundary problems with overdetermined boundary conditions were also investigated by Hamel and Nadirashvili; see Theorems 1.10 and 1.13 in \cite{HN} (see also \cite{Ruiz, Wang}). We would like to underline that non-stagnation conditions in these results are generally necessary. Ruiz \cite{Ruiz2} constructed a non-circular steady Euler flow in a simply connected domain with overdetermined boundary conditions, exhibiting at least two stagnation points within the domain. Similar nonsymmetric solutions in annular-shaped domains are also known in the literature; see, e.g., \cite{Ago, Kam}.

It is noteworthy that the problems addressed in the present paper are formulated within the context of exterior domains. It appears that there has been relatively limited exploration in this aspect. In comparison with the aforementioned work, the results of this paper can be regarded as a complementary addition in this context.

\section{Proofs for the main results}\label{s2}
This section is dedicated to the proof of Theorems \ref{th3}, \ref{th2} and \ref{th1}. The primary focus of the proof is to analyze the stream function of the flow and subsequently reformulate the problem into a symmetry issue regarding the positive solution of the semilinear elliptic equation with overdetermined boundary conditions.

In a celebrated paper \cite{Se}, Serrin inaugurates the investigation into the radial symmetry of solutions to elliptic equations subject to overdetermined boundary conditions. Serrin's approach is now known as the method of moving planes. It is a very powerful technique in proving symmetry results for positive solutions of elliptic and parabolic problems in symmetric domains. We refer the reader to the survey articles \cite{Cir, Ni, Nit, Pie, Si2} and the references therein for more extensive discussions on this and related topics.

  Several symmetry results for elliptic problems set in exterior domains were established in \cite{Aft, Rei2, Ros, Si}. However, the nonlinearity in these established results is often constrained by monotonicity conditions, which renders them inapplicable to our scenario. Our proof is based on a refinement of the method of moving planes, mainly by appropriately combining and developing the ideas in \cite{HN, Rei, Ruiz}.

Section \ref{s2.1} is devoted to some common lemmas. The proof of Theorem \ref{th3} is completed in Section \ref{s2.2}. In Section \ref{s2.3}, we present the proof of Theorem \ref{th2}. The proof of Theorem \ref{th1} is done in Section \ref{s2.4}

\subsection{Some common lemmas}\label{s2.1}
Let $\Omega\subset \R^2$ be a simply connected bounded domain of class $C^2$ and $n$ be the outward unit normal on $\partial \Omega$. We consider steady incompressible flows solving the Euler equations
\begin{equation}\label{2-1}
  \left\{
\begin{aligned}
  \mathbf{v}\cdot \nabla \mathbf{v} +\nabla P & =0&\text{in}\ \, \R^2\backslash \overline{\Omega},\\
  \nabla\cdot\mathbf{v} & =0\,\ \, \ \ \ \ \ \ \  \ \, &\text{in}\ \, \R^2\backslash \overline{\Omega},\\
   \mathbf{v}\cdot n &  =0\,\ \, \ \ \ \ \ \ \  \ \,&\text{on}\ \, \partial \Omega.\ \ \,\\
\end{aligned}
\right.
\end{equation}
The following result is concerned with the existence and some elementary properties of the stream function of the flow, which is essentially contained in \cite{HN} (see also \cite{Ruiz}).
\begin{lemma}[\cite{HN}]\label{le1}
  Let $\mathbf{v}$ be a $C^2({\R^2\backslash\Omega})$ flow solving \eqref{2-1}. Assume also that:
  \begin{itemize}
    \item [(1)]$|\mathbf{v}|>0$ in $\R^2\backslash \overline{\Omega}$.
    \item [(2)]The following far-field conditions hold
    \begin{equation*}
      \liminf_{|x|\to +\infty}v^\theta(x)>0\ \ \ \text{and}\ \ \ v^r(x)=o\left(\frac{1}{|x|} \right)\ \text{as}\ |x|\to +\infty.
    \end{equation*}
  \end{itemize}
  Then there is a unique (up to additive constants) stream function $u: \R^2\backslash \Omega\to \R$ of class $C^3(\R^2\backslash \Omega)$ such that
  \begin{equation*}
  \nabla^\perp u=\mathbf{v},\ \ \ \text{that is},\ \ \ \partial_1 u=v_2\ \ \text{and}\ \ \partial_2u=-v_1.
\end{equation*}
Moreover, we have that
\begin{equation}\label{2-5}
  u=0\ \, \text{on}\ \,\partial \Omega,\ \ \ 0<u<+\infty\ \,\text{in}\ \,\R^2\backslash\overline{\Omega},\ \ \ \text{and}\ \ \ \lim_{|x|\to +\infty}u(x)=+\infty.
\end{equation}
Furthermore, there exists a function $f\in C([0, +\infty)) \cap C^1(0, +\infty)$ such that
  \begin{equation}\label{2-2}
   \Delta u+f(u)=0\ \ \ \text{in}\ \ \R^2\backslash \Omega.
  \end{equation}
If, additionally, $\{x\in \partial \Omega: |\mv(x)|=0\}\subsetneq \partial \Omega$, then the function $f$ is of class $C^1[0, +\infty)$.
\end{lemma}

 For $\lambda\in \R$, we define
 \begin{equation*}
   \mathcal{T}_\lambda=\{x=(x_1, x_2)\in \R^2: x_1=\lambda\},\ \ \ \mathcal{H}_\lambda=\{x=(x_1, x_2)\in \R^2: x_1>\lambda\}.
 \end{equation*}
Let $\mathcal{R}_\lambda$ be the reflection with respect to $\mathcal{T}_\lambda$. For $x\in \R^2$, set $x^\lambda=\mathcal{R}_\lambda(x)=(2\lambda-x_1, x_2)$.

The following lemma states that the streamlines of the flow are almost circular at infinity (although we are only discussing the $x_1$-direction here, the result holds true for each direction); see Lemma 3.1 in \cite{HN}.

\begin{lemma}[\cite{HN}]\label{le2}
  Let $\mv$ be as in Lemma \ref{le1}. Consider any point $x\in \R^2\backslash \overline{\Omega}$. Let $\xi_x$ be the solution of
	\begin{align*}
		\begin{cases}
			 \dot{\xi}_x(t) = \mathbf{v}(\xi_x(t)), &\\
             \xi_x(0)=x.\ \ &
\end{cases}
	\end{align*}
Then $\xi_x$ is defined in $\mathbb{R}$ and periodic, and the streamline $\Phi_x:=\xi_x(\mathbb{R})$ is a $C^2$ Jordan curve surrounding $\overline{\Omega}$ in $ \R^2\backslash \overline{\Omega}$.

Let $D_x$ denote the bounded connected component of $\R^2\backslash \Phi_x$. Then for each $\varepsilon>0$, there exists $R_\varepsilon>0$ such that
\begin{equation*}
  \mathcal{R}_\lambda \left(\mathcal{H}_\lambda \cap \overline{D_x} \right) \subset D_x
\end{equation*}
for all $\lambda>\varepsilon$ and $|x|\ge R_\varepsilon$.
\end{lemma}

The following result is on the behavior of the stream function of the flow near the boundary of $\Omega$, which serves as an analogue of Proposition 2 in \cite{Rei}.
\begin{lemma}\label{le4}
Let $\mv, u, f$ be as in Lemma \ref{le1}. Assume also that one of the following conditions holds:
\begin{itemize}
  \item [(i)]$\{x\in \partial \Omega: |\mv(x)|=0\}\subsetneq \partial \Omega$.
  \item [(ii)]The vorticity $\omega=\partial_1v_2-\partial_2 v_1$ does not vanish on $\partial \Omega$.
    \item [(iii)]$\Omega=B_R$ for some $R>0$. In addition, there exists a $\delta>0$ such that $v^\theta(x)>0$ for all $R<|x|<R+\delta$, and $v^r(x)=o\left(v^\theta(x)\right)$ as $|x|\to R^+$.

\end{itemize}
 For $z\in \partial \Omega$ let $\nu$ be an inward unit direction into $\R^2\backslash \overline{\Omega}$, i.e. $n(z)\cdot \nu>0$. Then there exists a disk $B_\rho (z)$ such that $\partial_\nu u>0$ in $B_\rho(z)\backslash \overline{\Omega}$.
\end{lemma}

\begin{proof}
If condition (i) holds, then $f$ in \eqref{2-2} is of class $C^1[0, +\infty)$. The conclusion can be readily derived by mimicking the proof of Proposition 2 in \cite{Rei}. Furthermore, if condition (ii) holds, then $f(0)<0$ due to \eqref{2-5} and the fact that $\Delta u=-f(u)$. One can easily obtain the desire results by obvious modification to the proof of Proposition 2 in \cite{Rei} (see in particular Case 2).

Now, let us assume that condition (iii) holds. Note that
\begin{equation*}
  n(z)=\mathbf{e}_r(z)\ \ \ \text{and}\ \ \ \nabla u(x)=v^\theta(x)\mathbf{e}_r(x)-v^r(x) \mathbf{e}_\theta(x).
\end{equation*}
Let us express $\nu$ as $\nu=\nu^r(x)\mathbf{e}_r(x)+\nu^\theta (x) \mathbf{e}_\theta(x)$. It follows from the assumption $\nu \cdot \mathbf{e}_r(z)>0$ that $\nu^r(x)$ is positive and away from zero in some neighborhood of $z$. We conclude that
\begin{equation*}
\begin{split}
    \partial_\nu u(x) & = \left(v^\theta(x)\mathbf{e}_r(x)-v^r(x) \mathbf{e}_\theta(x)\right)\cdot \left(\nu^r(x)\mathbf{e}_r(x)+\nu^\theta (x) \mathbf{e}_\theta(x)\right) \\
     & =v^\theta(x)\nu^r(x)-v^r(x)\nu^\theta(x)>0
\end{split}
\end{equation*}
when $x$ is sufficiently close to $z$, since $ v^r(x)=o\left(v^\theta(x)\right)$ as $|x|\to R^+$ by assumption.
\end{proof}

\begin{remark}
  We may assume $\rho$ so small, that $\nu$ is an inward unit direction into $\R^2\backslash \overline{\Omega}$ for all boundary points in $\partial \Omega \cap B_\rho (z)$.
\end{remark}


The following lemma provides a comparative analysis of the stream function of the flow, analogous to Proposition 1.14 in \cite{HN}.

\begin{lemma}\label{le5}
  Suppose the assumptions of Lemma \ref{le4} hold. Let $\varepsilon\ge0$. Assume also that
  \begin{equation}\label{2-7}
    \mathcal{R}_\lambda\left(\mathcal{H}_\lambda \cap \overline{\Omega} \right) \subset \Omega\ \ \ \text{for all}\ \lambda>\varepsilon.
  \end{equation}
  Then, for every $\lambda\ge \varepsilon$, there holds
  \begin{equation*}
    u(x)\ge u(x^\lambda)\ \ \ \text{for all}\ x\in \mathcal{H}_\lambda \backslash \mathcal{R}_\lambda(\overline{\Omega}).
  \end{equation*}
\end{lemma}

\begin{proof}
The proof is based on the method of moving planes developed in \cite{Ale, Bre, Gid, Se}. We adopt a similar approach to that presented in \cite{HN, Rei}.

The continuity of $u$ implies that it is sufficient to consider the case when $\varepsilon>0$. For $\lambda>\varepsilon$, let us introduce the comparison function
  \begin{equation}\label{2-39}
    \Phi_\lambda(x):=u(x^\lambda)-u(x),\ \ \ x\in \mathcal{H}_\lambda \backslash \mathcal{R}_\lambda(\overline{\Omega}),
  \end{equation}
  which is well-defined by virtue of \eqref{2-7}. Our task now is to show that for all $\lambda>\varepsilon$, there holds $\Phi_\lambda\le 0$ in $\mathcal{H}_\lambda \backslash \mathcal{R}_\lambda(\overline{\Omega})$. In view of Lemmas \ref{le1} and \ref{le2}, we have that
  \begin{equation*}
    \lim_{|x|\to +\infty} \min_{\R}|\xi_{x}|=+\infty.
  \end{equation*}
Moreover, there exists $R_\varepsilon>0$ such that
\begin{equation*}
  \mathcal{R}_\lambda \left(\mathcal{H}_\lambda \cap \overline{D_z} \right) \subset D_z
\end{equation*}
for all $\lambda>\varepsilon$ and $|z|\ge R_\varepsilon$. Clearly, it suffices to show that for all $|z|>R_\varepsilon$ and for all $\lambda>\varepsilon$, there holds
\begin{equation}\label{2-12}
  \Phi_\lambda \le 0\ \ \ \text{in}\ \left(\mathcal{H}_\lambda\cap (D_{z}\backslash \overline{\Omega})\right)\backslash \mathcal{R}_\lambda(\overline{\Omega}).
\end{equation}
We will maintain the assumption that $|z|>R_\varepsilon$ for the rest of the argument. To simplify the notation, let us denote
 \begin{equation*}
   \Sigma_\lambda=\left(\mathcal{H}_\lambda\cap (D_{z}\backslash \overline{\Omega})\right)\backslash \mathcal{R}_\lambda(\overline{\Omega}).
 \end{equation*}
 Call $M_0=\max_{x\in D_{z}}x_1>0$. Our plan is to show the following properties for $\Phi_\lambda$ for all $\lambda\in(\varepsilon, M_0)$:
\begin{align}
\label{2-13}  \Phi_\lambda\le 0 &\ \ \text{in}\ \Sigma_\lambda, \\
\label{2-14}  \partial_1 \Phi_\lambda <0 &\ \ \text{on}\ \mathcal{T}_\lambda\cap (D_{z}\backslash \overline{\Omega}).
\end{align}
This will be done by an initial step for $\lambda\in (M_0-\tau, M_0)$ with $\tau>0$ small and by a continuation step for all $\lambda\in(\varepsilon, M_0)$.

  Let us start with the \textbf{initial step}: Recalling \eqref{2-2}, we have that
  \begin{equation*}
    \Delta \Phi_\lambda+c_\lambda \Phi_\lambda =0\ \ \ \text{in}\ \overline{\mathcal{H}_\lambda \backslash \mathcal{R}_\lambda(\overline{\Omega})},
  \end{equation*}
  where
  \begin{equation}\label{2-40}
    c_\lambda(x)=\begin{cases}
                   \displaystyle\frac{f(u(x^\lambda))-f(u(x))}{u(x^\lambda)-u(x)}, & \mbox{if}\ \ u(x^\lambda)\not=u(x), \\
                   0, & \mbox{if}\  \ u(x^\lambda)=u(x).
                 \end{cases}
  \end{equation}
For $\lambda$ less than $M_0$ and close to $M_0$, we have that $\overline{\Sigma_\lambda \cup \mathcal{R}_\lambda (\Sigma_\lambda)} \subset \R^2 \backslash \overline{\Omega} $, and hence $c_\lambda\in L^\infty(\overline{\Sigma_\lambda})$. Notice that $\Phi_\lambda \le 0$ on $\partial \Sigma_\lambda$. By the maximum principle in sets with bounded diameter and small Lebesgue measure, and then from the strong maximum principle (see, e.g., \cite{Fra}), we get that $\Phi_\lambda<0$ in $\Sigma_\lambda$. Furthermore, considering that $\Phi_\lambda(x)=0$ on $\mathcal{T}_\lambda\cap (D_{z}\backslash \overline{\Omega})$, by the Hopf Lemma, we observe that $\partial_1\Phi_\lambda(x)=-2\partial_1u(x)<0$ for all $x\in\mathcal{T}_\lambda\cap (D_{z}\backslash \overline{\Omega})$.

\textbf{Continuation step}: By the initial step the following quantity $\lambda_*$ is well defined
\begin{equation*}
  \lambda_*=\inf\{\lambda\in (\varepsilon, M_0):\, \Phi_{\lambda'}<0\ \text{in}\ \Sigma_{\lambda'}\ \text{for all}\ \lambda'\in(\lambda, M_0)\}.
\end{equation*}
Moreover, by the maximum principle, we see that both \eqref{2-13} and \eqref{2-14} hold for all $\lambda\in (\lambda_*, M_0)$. Our intention is to show $\lambda_*=\varepsilon$. Suppose $0<\lambda_*<M_0$. By the assumption \eqref{2-7}, the $x_1$-direction is non-tangent on $\partial \Sigma_{\lambda_*}$; see Lemma A.1 in \cite{Ami}. Since $u$ is constant along $\partial D_z$ and $|\nabla u| \neq 0$ on $\partial D_z$, we have $\partial_1 u > 0$ on $\mathcal{T}_\lambda \cap \partial D_z$. There are sequences $\lambda_k \nearrow \lambda_*$ and $x^{(k)}\in \Sigma_{\lambda_k}$ such that $\Phi_{\lambda_k}$ attains its positive maximum over $\Sigma_{\lambda_k}$ in $x^{(k)}$ and $x^{(k)}\to \bar{x}$ with $\Phi_{\lambda_*}(\bar{x})=0$ and $\nabla \Phi_{\lambda_*}(\bar{x})=0$. Note that $\Phi_\lambda(x)<0$ for all $x\in\mathcal{H_{\lambda_*}}\cap \left(\partial D_z \cup \partial \mathcal{R}_{\lambda_*}(\Omega) \right)$. Combined this fact with the Hopf Lemma, we see that $\bar{x}\in \mathcal{T}_{\lambda_*}\cap (\partial D_z \cup \partial \Omega)$ and $x^{(k)}, x^{(k), \lambda_k}\to \bar{x}$. By Lemma \ref{le4}, in the vicinity of $\bar{x}$ we have $\partial_1u>0$. Integration of $\partial_1 u$ along straight lines from $x^{(k)}$ to $x^{(k), \lambda_k}$ yields $\Phi_{\lambda_k}(x^{(k)})<0$ for $k$ big enough, a contradiction. This shows that $\lambda_*=\varepsilon$, and thus completes the proof of this case.
\end{proof}

We have the following basic estimate.
\begin{lemma}\label{le7}
Let $\mv, u, f$ be as in Lemma \ref{le1}. Assume also that $|v|$ is constant on $\partial \Omega$  and the vorticity $\omega=\partial_1v_2-\partial_2v_1$ does not vanish on $\partial \Omega$. Let $\delta>0$ and $\Omega_\delta=\{x\in \R^2\backslash \overline{\Omega}: d(x)< \delta \}$, where $d(x)=d(x, \partial \Omega)=\min\{|x-p|: p\in\partial \Omega\}$. Then there exists $C>0$ such that
\begin{equation}\label{2-20}
  \left|f(u(x))-f(u(y))\right|\le \frac{C}{\min\{d(x), d(y)\}}|u(x)-u(y)|,\ \ \ \forall\, x, y\in \Omega_\delta.
\end{equation}
\end{lemma}
The proof of Lemma \ref{le7} is straightforward. In fact, it closely resembles the proof of Lemma 5.2 in \cite{Ruiz}, thus it is omitted.

The following two results are attributed to Ruiz \cite{Ruiz}. The first is the Hopf lemma for singular operators.
\begin{lemma}[\cite{Ruiz}, Proposition 4.3]\label{le8}
  Let $N\ge 2$ be an integer. Let $D\subset \R^N$ be a $C^2$ domain, and $c:D\to \R$ satisfy that $c(x)d(x)\in L^\infty(D)$, where $d(x)=d(x, \partial D)=\min\{|x-p|: p\in\partial D\}$. Let $B_r\subset D$ be a ball of radius $r>0$, and $\psi\in C^1(\overline{B_r})$ solving
  \begin{equation*}
    -\Delta \psi +c(x)\psi \ge 0
  \end{equation*}
  in a weak sense, that is, for any $\varphi\in H_0^1(D)$, $\varphi\ge 0$,
  \begin{equation*}
    \int_D \nabla \psi \cdot \nabla \varphi+c(x)\psi \varphi \ge 0.
  \end{equation*}
  Let $\nu$ be the outward unit normal on $\partial D$.
  Assume that
  \begin{itemize}
    \item [(1)] $\psi\ge 0$ in $B_r$;
    \item [(2)] $\psi(p)=0$ for some $p\in \partial B_r$.
  \end{itemize}
  Then
  \begin{equation*}
    \text{either}\ \ \partial_\nu \psi(p)<0\ \ \text{or}\ \ \psi=0\ \text{in}\ B_r.
  \end{equation*}
\end{lemma}

The next lemma is the Serrin corner lemma for singular operators.
\begin{lemma}[\cite{Ruiz}, Proposition 4.4]\label{le9}
   Let $N\ge 2$ be an integer. Let $D\subset \R^N$ be a $C^2$ domain, and $c:D\to \R$ satisfy that $c(x)d(x)\in L^\infty(D)$, where $d(x)=d(x, \partial D)=\min\{|x-p|: p\in\partial D\}$. Let $B_r\subset D$ be a ball of radius $r>0$, and $B_r^+$ a half ball. We can assume, without loss of generality, that
   \begin{equation*}
     B_r^+=\left\{x\in\R^N: |x|<r,\ x_1>0 \right\}.
   \end{equation*}
   Let $\psi\in C^2(\overline{B^+_r})$ be a weak solution of the inequality
   \begin{equation*}
      -\Delta \psi +c(x)\psi \ge 0.
   \end{equation*}
   Assume that
   \begin{itemize}
     \item [(1)] $\psi \ge 0$ in $B_r^+$;
     \item [(2)]$\psi(p)=0$ for some $p\in \partial B_r \cap \{x_1=0\}$;
     \item [(3)]$\nabla \psi(p)=0$.
   \end{itemize}
     Then
  \begin{equation*}
    \text{either}\ \ \frac{\partial^2 \psi}{\partial \eta^2}(p)>0\ \ \text{or}\ \ \psi=0\ \text{in}\ B_r^+,
  \end{equation*}
  where $\eta\in \R^N$ is any unit vector with $\eta_1>0$, $p\cdot \eta<0$.
\end{lemma}

\subsection{Proof of Theorem \ref{th3}}\label{s2.2}
This subsection is devoted to the proof of Theorem \ref{th3}. Let $\Omega\subset \R^2$ be a simply connected bounded domain of class $C^2$ and $n$ be the outward unit normal on $\partial \Omega$. Let $\mathbf{v}$ be a $C^2({\R^2\backslash\Omega})$ flow solving \eqref{1-1}. Assume also that:
  \begin{itemize}
    \item [(1)]$|\mathbf{v}|>0$ in $\R^2\backslash \overline{\Omega}$ and $|\mathbf{v}|$ is a nonzero constant on $\partial \Omega$.
    \item [(2)]The following far-field conditions hold
    \begin{equation}\label{2-20}
      \liminf_{|x|\to +\infty}|\mathbf{v}(x)|>0\ \ \ \text{and}\ \ \ \mathbf{v}(x)\cdot \mathbf{e}_r(x)=o\left(\frac{1}{|x|} \right)\ \text{as}\ |x|\to +\infty.
    \end{equation}
  \end{itemize}
  Our goal is to show that $\mathbf{v}$ is a circular flow with respect to the origin. In fact, if $\mv$ is a circular flow with respect to some point $z\in \R^2$, then $z$ must be the origin due to the far-field conditions \eqref{2-20}. Moreover, given the rotational invariance of the problem at hand, establishing radial symmetry only necessitates demonstrating symmetry with respect to one direction, for instance, the $x_1$-direction. In addition, we may assume, without loss of generality, that
      \begin{equation*}
      \liminf_{|x|\to +\infty}v^\theta(x)>0\ \ \ \text{and}\ \ \ v^r(x)=o\left(\frac{1}{|x|} \right)\ \text{as}\ |x|\to +\infty
    \end{equation*}
   after possibly changing $\mv$ into $-\mv$. Therefore, the conditions of Lemma \ref{le1} are fulfilled in this case. In particular, the nonlinearity $f$ in \eqref{2-2} now belongs to the class $C^1[0, +\infty)$.

\textbf{Step 1:} Set $R_0=\max_{x\in \partial \Omega} x_1$. It is well known (see Lemma A.1 in \cite{Ami}) that for values of $\lambda$ a little less than $R_0$ the reflection of $\mathcal{H}_\lambda\cap \overline{\Omega}$ is lying in $\Omega$ and the $x_1$-direction is external at every point of $\mathcal{H}_\lambda \cap \partial \Omega$. Moreover, this stays true for decreasing values of $\lambda$ until one of the following alternatives happens:
  \begin{itemize}
    \item [(i)]\textit{Internal tangency}. There exists $p\in \mathcal{H}_\lambda \cap \partial \Omega$ such that $\mathcal{R}_\lambda(p)\in \partial \Omega$.
    \item [(ii)]\textit{Orthogonality of $\partial \Omega$ and $\mathcal{T}_\lambda$}. There exists $p\in \partial \Omega\cap \mathcal{T}_\lambda$ such that $n_1(p)=0$.
  \end{itemize}
We denote this critical value of $\lambda$ by $\mu$. Without loss of generality, we may assume $\mu \geq 0$ by considering the $-x_1$-direction instead of the $x_1$-direction if necessary. Hence, there holds
\begin{equation*}
    \mathcal{R}_\lambda\left(\mathcal{H}_\lambda \cap \overline{\Omega} \right) \subset \Omega\ \ \ \text{for all}\ \lambda>\mu.
  \end{equation*}
According to Lemma \ref{le5}, we deduce that $\Phi_\mu \le 0$ in $\mathcal{H}_\mu \backslash \mathcal{R}_\mu(\overline{\Omega})$, where $\Phi_\mu$ is the comparison function defined by \eqref{2-39}. We remark that in order to establish the $x_1$-symmetry of $\Omega$, it suffices to show that $\Phi_\mu\equiv 0$ on a component $Z$ of $\mathcal{H}_\mu \backslash \mathcal{R}_\mu(\overline{\Omega})$. For a detailed proof of this fact, please refer to pages 385-386 in \cite{Rei2}.

\textbf{Step 2:} In this step, the Neumann boundary condition will play an important role. We shall adopt the same procedure as in the proof of Theorem 5.1 in \cite{Ruiz}.

\textit{Case 1: Internal tangency.} Assume that we have internal tangency at a point $p\in \mathcal{H}_\lambda \cap \partial \Omega$, for instance. By the overdetermined boundary condition,
\begin{equation}\label{3-20}
  \partial_n \Phi_{\mu}(p)=0.
\end{equation}
Due to the $C^2$ regularity of the domain, we are able to select an open disk $B_r(z)$ in $\mathbb{R}^2 \setminus \overline{\Omega}$, tangent to $\partial\Omega$ at the point $p$. We can shrink that disk such that $B_r(z) \subset \mathcal{H}_\mu \backslash \mathcal{R}_\mu(\overline{\Omega})$. In other words, $B_r(z) \cap \left( \overline{\Omega}\cup \mathcal{R}_\mu(\overline{\Omega}) \right)=\varnothing$. Recall that
  \begin{equation}\label{4-1}
    \Delta \Phi_\mu+c_\mu \Phi_\mu =0\ \ \ \text{in}\ \overline{\mathcal{H}_\mu \backslash \mathcal{R}_\mu(\overline{\Omega})},
  \end{equation}
  where $c_\mu\in L^\infty(B_r(z))$ is defined by \eqref{2-40}. We proceed by applying Lemma \ref{le8} to the domain $B_r(z)$, along with \eqref{3-20}. This allows us to infer initially that $\Phi_\mu$ is identically zero in $B_r(z)$, and subsequently, $\Phi_\mu$ remains identically zero in the component $Z$ of $\mathcal{H}_\mu \backslash \mathcal{R}_\mu(\overline{\Omega})$ containing $B_r(z)$. This implies that $\Omega$ is symmetric with respect to $\mathcal{T_\mu}$.

\textit{Case 2: Orthogonality of $\partial \Omega$ and $\mathcal{T}_\mu$.} Let us consider the case of a certain point $p\in \mathcal{T}_\mu \cap \partial \Omega $ with $n_1(p)=0$. Clearly, it holds that $\nabla \Phi_\mu (p)=0$. Reasoning as in pages 389-391 in \cite{Rei2} (see also \cite{Rei}), we conclude that the second order derivatives of $\Phi_\mu$ at $p$ are zero:
\begin{equation}\label{3-23}
  D^2 \Phi_\mu (p)=0.
\end{equation}
As in the previous case, we can take an open disk $B_r(z)$ in $\mathbb{R}^2 \setminus \overline{\Omega}$, tangent to $\partial\Omega$ at the point $p$. We can shrink $B_r(z)$ such that $B_r(z) \cap \left( \overline{\Omega}\cup \mathcal{R}_\mu(\overline{\Omega}) \right)=\varnothing$. We now apply Lemma \ref{le9} to the domain $B_r(z)$ and $B_r^+(z)$: taking into account \eqref{3-23} we conclude that $\Phi_\mu\equiv 0$ in $B_r^+(z)$. As in the previous case, we can further deduce that $\Phi_\mu$ remains identically zero in the component $Z$ of $\mathcal{H}_\mu \backslash \mathcal{R}_\mu(\overline{\Omega})$ containing $B_r(z)$, and hence $\Omega$ is symmetric with respect to $\mathcal{T_\mu}$.

In both cases, we conclude that $\Omega$ is symmetric with respect to the $x_1$-direction. The proof of Theorem \ref{th3} is thus complete.

\subsection{Proof of Theorem \ref{th2}}\label{s2.3}

This subsection is devoted to the proof of Theorem \ref{th2}. Let $\mathbf{v}$ be a $C^2({\R^2\backslash\Omega})$ flow solving \eqref{1-1}. Assume also that:
  \begin{itemize}
    \item [(1)]$|\mathbf{v}|>0$ in $\R^2\backslash \overline{\Omega}$ and $|\mv|=0$ on $\partial \Omega$.
       \item [(2)]The following far-field conditions hold
    \begin{equation*}
      \liminf_{|x|\to +\infty}|\mathbf{v}(x)|>0\ \ \ \text{and}\ \ \ \mathbf{v}(x)\cdot \mathbf{e}_r(x)=o\left(\frac{1}{|x|} \right)\ \text{as}\ |x|\to +\infty.
    \end{equation*}
    \item [(3)]The vorticity $\omega=\partial_1v_2-\partial_2 v_1$ does not vanish on $\partial \Omega$.
  \end{itemize}
  Our goal is to show that $\mathbf{v}$ is a circular flow with respect to the origin. The proof follows the same argument as in the previous subsection, so we will be sketchy. Firstly, Step 1 in the proof of Theorem \ref{th3} clearly remains valid. The only aspect requiring appropriate modification is Step 2, mainly because $c_\mu$ in \eqref{4-1} is no longer in $L^\infty$. Nevertheless, we notice that Lemmas \ref{le8} and \ref{le9} only require that $c(x)d(x)\in L^\infty$. Therefore, we make the following modifications:

\textit{Case 1: Internal tangency.} Assume that we have internal tangency at a point $p\in \mathcal{H}_\lambda \cap \partial \Omega$, for instance. By the overdetermined boundary condition,
\begin{equation}\label{4-20}
  \partial_n \Phi_{\mu}(p)=0.
\end{equation}
Due to the $C^2$ regularity of the domain, we are able to select an open disk $B_r(z)$ in $\mathbb{R}^2 \setminus \overline{\Omega}$, tangent to $\partial\Omega$ at the point $p$. We can shrink that disk such that $B_r (z)\subset \mathcal{H}_\mu \backslash \mathcal{R}_\mu(\overline{\Omega})$. In other words, $B_r(z) \cap \left( \overline{\Omega}\cup \mathcal{R}_\mu(\overline{\Omega}) \right)=\varnothing$. From this, we have that
\begin{equation}\label{4-21}
  d(x, \partial \Omega)\ge r-|x|,\ \ \ d(x, \partial \Omega)\ge r-|x|\ \ \ \text{for any}\ x\in B_r(z).
\end{equation}
Recall that
  \begin{equation*}
  \Delta \Phi_\mu+c_\mu \Phi_\mu =0\ \ \ \text{in}\ \overline{\mathcal{H}_\mu \backslash \mathcal{R}_\mu(\overline{\Omega})},
  \end{equation*}
  where $c_\mu$ is defined by \eqref{2-40}. Combining \eqref{4-21} and Lemma \ref{le7}, we conclude that
\begin{equation}\label{4-22}
  |c_\mu(x)|\le \frac{C}{r-|x|}\ \ \ \text{in}\ B_r(z).
\end{equation}
We proceed by applying Lemma \ref{le8} to the domain $B_r(z)$, along with \eqref{4-20}. This allows us to infer initially that $\Phi_\mu$ is identically zero in $B_r(z)$, and subsequently, $\Phi_\mu$ remains identically zero in the component $Z$ of $\mathcal{H}_\mu \backslash \mathcal{R}_\mu(\overline{\Omega})$ containing $B_r(z)$. This implies that $\Omega$ is symmetric with respect to $\mathcal{T_\mu}$.

\textit{Case 2: Orthogonality of $\partial \Omega$ and $\mathcal{T}_\mu$.} Let us consider the case of a certain point $p\in \mathcal{T}_\mu \cap \partial \Omega $ with $n_1(p)=0$. Clearly, it holds that $\nabla \Phi_\mu (p)=0$. Reasoning as in pages 389-391 in \cite{Rei2} (see also \cite{Rei}), we conclude that the second order derivatives of $\Phi_\mu$ at $p$ are zero:
\begin{equation}\label{4-23}
  D^2 \Phi_\mu (p)=0.
\end{equation}
As in the previous case, we can take an open disk $B_r(z)$ in $\mathbb{R}^2 \setminus \overline{\Omega}$, tangent to $\partial\Omega$ at the point $p$. We can shrink $B_r(z)$ such that $B_r(z) \cap \left( \overline{\Omega}\cup \mathcal{R}_\mu(\overline{\Omega}) \right)=\varnothing$. As a consequence, also here \eqref{4-21} is satisfied, and then the estimate \eqref{4-22} holds. We now apply Lemma \ref{le9} to the domain $B_r(z)$ and $B_r^+(z)$: taking into account \eqref{4-23} we conclude that $\Phi_\mu\equiv 0$ in $B_r^+(z)$. As in the previous case, we can further deduce that $\Phi_\mu$ remains identically zero in the component $Z$ of $\mathcal{H}_\mu \backslash \mathcal{R}_\mu(\overline{\Omega})$ containing $B_r(z)$, and hence $\Omega$ is symmetric with respect to $\mathcal{T_\mu}$.

In both cases, we conclude that $\Omega$ is symmetric with respect to the $x_1$-direction. The proof of Theorem \ref{th2} is thus complete.

\subsection{Proof of Theorem \ref{th1}}\label{s2.4}
In this subsection, we present the proof of Theorem \ref{th1}. Let $\mathbf{v}$ be a $C^2(\R^2)$ flow solving \eqref{1-5}. Assume also that:
  \begin{itemize}
    \item [(1)]$\mathcal{S}=B_R$ for some $R>0$.
  \item [(2)]The following far-field conditions hold
    \begin{equation*}
          \liminf_{|x|\to +\infty}v^\theta(x)>0\ \ \ \text{and}\ \ \ v^r(x)=o\left(\frac{1}{|x|} \right)\ \text{as}\ |x|\to +\infty.
    \end{equation*}
       \item [(3)]There exists a $\delta>0$ such that $v^\theta(x)>0$ for all $R<|x|<R+\delta$, and $v^r(x)=o\left(v^\theta(x)\right)$ as $|x|\to R^+$.
  \end{itemize}
  Our goal is to show that $\mathbf{v}$ is a circular flow with respect to the origin. It suffices to show that the stream function $u$ is radially symmetric. Since the problem under consideration is rotationally invariant, it suffices for the proof of radial symmetry to establish the symmetry with respect to one coordinate axis, e.g., the  $x_1$-axis. By Lemma \ref{le5}, we have that
  \begin{equation}\label{3-1}
    u(x_1, x)\ge u(-x_1, x)\ \ \ \text{for all}\ x=(x_1, x_2)\in \R^2 \backslash \overline{B_R},\ x_1>0.
  \end{equation}
   By using a rotation which replaces the $x_1$-axis of our coordinate frame by the $-x_1$-axis, we can then deduce equality in \eqref{3-1}, i.e., symmetry in the $x_1$-direction. The proof of Theorem \ref{th1} is thereby complete.

\subsection*{Acknowledgments}
 D. Cao and B. Fan were supported by National Key R \& D Program (2023YFA1010001) and NNSF of China (Grant No.\,12371212). W. Zhan was supported by NNSF of China (Grant No.\,12201525).

\subsection*{Data Availability} Data sharing is not applicable to this article as no datasets were generated or analyzed during the current study.

\subsection*{Declarations}

\smallskip
\ \ \\

\noindent \textbf{Conflict of interest} The authors declare that they have no conflict of interest.

	\phantom{s}
	\thispagestyle{empty}


\begin{thebibliography}{99}

\bibitem{Aft}
A. Aftalion and J. Busca, Radial symmetry of overdetermined boundary value problems in exterior domains, \textit{Arch. Ration. Mech. Anal.}, 14 (1998), 195--206.


\bibitem{Ago}
V. Agostiniani, S. Borghini and L. Mazzieri, On the Serrin problem for ring-shaped domains, to appear in \textit{J. Eur. Math. Soc. (JEMS)}, arXiv:2109.11255.


\bibitem{Ale}
A. D. Alexandrov, A characteristic property of the spheres, \textit{Ann. Mat. Pura Appl.}, 58 (1962), 303--354.

\bibitem{Ami}
C. J. Amick and L. E. Fraenkel, The uniqueness of Hill's spherical vortex, \textit{Arch. Rational Mech. Anal.}, 92 (1986), no. 2, 91--119.



\bibitem{Bed}
J. Bedrossian, N. Masmoudi, Inviscid damping and the asymptotic stability of planar shear flows in
the 2D Euler equations, \textit{Publ. Math. Inst. Hautes \'{E}tudes Sci.}, 122 (2015), 195--300.

\bibitem{Bre}
H. Berestycki, L. A. Caffarelli and L. Nirenberg, Monotonicity for elliptic equations in unbounbded Lipschitz domains, \textit{Comm. Pure Appl. Math.}, 50 (1997), 1089--1111.


\bibitem{Bro0}
 F. Brock, Continuous Steiner-symmetrization, \textit{Math. Nachr.}, 172 (1995), 25--48.


\bibitem{Bro1}
 F. Brock, Continuous rearrangement and symmetry of solutions of elliptic problems, \textit{Proc. Indian Acad. Sci. Math. Sci.}, 110 (2000), no. 2, 157--204.

\bibitem{Cir}
G. Ciraolo and A. Roncoroni, The method of moving planes: a quantitative approach. \textit{Bruno Pini Mathematical Analysis Seminar}, 2018, 41--77, Bruno Pini Math. Anal. Semin., 9, \textit{Univ. Bologna, Alma Mater Stud., Bologna}, 2018.

\bibitem{Cons}
P. Constantin, T. D. Drivas and D. Ginsberg, Flexibility and rigidity in steady fluid motion, \textit{Comm. Math. Phys.}, 385 (2021), no. 1, 521--563.

\bibitem{Coti}
M. Coti Zelati, T. M. Elgindi and K. Widmayer, Stationary structures near the Kolmogorov and Poiseuille flows in the 2d Euler equations, \textit{Arch. Ration. Mech. Anal.}, 247 (2023), no. 1, Paper No. 12, 37 pp.



\bibitem{Fra}
L. E. Fraenkel, An introduction to maximum principles and symmetry in elliptic problems. Cambridge Tracts in Mathematics, 128. Cambridge University Press, Cambridge, 2000. x+340 pp. ISBN: 0-521-46195-2.

\bibitem{Gid}
B. Gidas, W. M. Ni and L. Nirenberg, Symmetry and related properties via the maximum principle, \textit{Comm. Math. Phys.}, 68 (1979), 209--243.


\bibitem{HN0}
F. Hamel and A. Karakhanyan, Potential flows away from stagnation in infinite cylinders, 2023, Preprint hal-04320798.


    \bibitem{HN3}
F. Hamel and N. Nadirashvili, Shear flows of an ideal fluid and elliptic equations in unbounded domains, \textit{Comm. Pure Appl. Math.}, 70 (2017), no. 3, 590–608.

    \bibitem{HN2}
F. Hamel and N. Nadirashvili, Parallel and circular flows for the two-dimensional Euler equations, Semin. Laurent Schwartz EDP Appl. 2017–2018, exp. V, 1--13.

    \bibitem{HN1}
    F. Hamel and N. Nadirashvili, A Liouville theorem for the Euler equations in the plane, \textit{Arch. Ration. Mech. Anal.}, 233 (2019), no. 2, 599--642.




    \bibitem{HN}
    F. Hamel and N. Nadirashvili, Circular flows for the Euler equations in two-dimensional annular domains, and related free boundary problems, \textit{J. Eur. Math. Soc. (JEMS)}, 25 (2023), no. 1, 323--368.


\bibitem{Ion}
A. D. Ionescu and H. Jia, Nonlinear inviscid damping near monotonic shear flows, \textit{Acta Math.}, 230 (2023), 321--399.


\bibitem{Kal}
H. Kalisch, Nonexistence of coherent structures in two-dimensional inviscid channel flow, \textit{Math. Model. Nat. Phenom.}, 7 (2012), 77--82.

\bibitem{Kam}
N. Kamburov and L. Sciaraffia, Nontrivial solutions to Serrin's problem in annular domains, \textit{Ann. Inst. H. Poincaré C Anal. Non Linéaire}, 38 (2021), no. 1, 1--22.



\bibitem{Li}
C. Li, Y. Lv, H. Shahgholian and C. Xie, Analysis of the steady Euler flows with stagnation points in an infinitely long nozzle, Preprint arXiv:2203.08375v3.



\bibitem{Ni}
W.-M. Ni, Qualitative properties of solutions to elliptic problems, Stationary partial differential equations. Vol. I, Handb. Differ. Equ., North-Holland, Amsterdam, 2004, pp. 157--233.



\bibitem{Nit}
C. Nitsch and C. Trombetti, The classical overdetermined Serrin problem, \textit{Complex Var. Elliptic Equ.}, 63 (2018), no. 7--8, 1107--1122.






\bibitem{Pie}
S. Pieralberto, A short survey on overdetermined elliptic problems in unbounded domains, \textit{Current trends in analysis, its applications and computation}, 451--461, Trends Math. Res. Perspect., Birkhäuser/Springer, Cham, 2022.

\bibitem{Rei}
W. Reichel, Radial symmetry by moving planes for semilinear elliptic boundary value problems on annuli and other nonconvex domains. In: Bandle, C., et al. (eds.) Progress in partial differential equations, elliptic and parabolic problems. Pitman Research Notes, vol. 325 (1995), 164--182.

\bibitem{Rei2}
W. Reichel, Radial symmetry for elliptic boundary-value problems on exterior domains, \textit{Arch. Ration. Mech. Anal.}, 137 (1997), 381--394.

\bibitem{Ros}
A. Ros, D. Ruiz and P. Sicbaldi, Solutions to overdetermined elliptic problems in nontrivial exterior domains, \textit{J. Eur. Math. Soc. (JEMS)}, 22 (2020), no. 1, 253--281.

\bibitem{Ruiz}
D. Ruiz, Symmetry results for compactly supported steady solutions of the 2D Euler equations, \textit{Arch. Ration. Mech. Anal.}, 247 (2023), no. 3, Paper No. 40, 25 pp.

\bibitem{Ruiz2}
D. Ruiz, Nonsymmetric sign-changing solutions to overdetermined elliptic problems in bounded domains, Preprint arXiv:2211.14014v2.

\bibitem{Se}
J. Serrin, A symmetry problem in potential theory, \textit{Arch. Ration. Mech. Anal.}, 43 (1971), 304--318.




\bibitem{Si}
B. Sirakov, Symmetry for exterior elliptic problems and two conjectures in potential theory, \textit{Ann. Inst. H. Poincaré Anal. Non Linéaire}, 18 (2001), 135--156.


\bibitem{Si2}
B. Sirakov, Overdetermined elliptic problems in physics. In: Nonlinear PDEs in Condensed Matter and Reactive Flows, Kluwer, 2002, 273--295.




\bibitem{Wang}
Y. Wang and W. Zhan, On the rigidity of the 2D incompressible Euler equations, Preprint arXiv:2307.00197v1.

\bibitem{Zill}
C. Zillinger, Linear inviscid damping for monotone shear flows, \textit{Trans. Amer. Math. Soc.}, 369 (2017), 8799--8855.

	\end{thebibliography}
\end{document}